\documentclass[11pt]{amsart}
\usepackage{amsfonts,amssymb,amscd,amsmath,enumerate,verbatim,calc}

\def\NZQ{\mathbb}               

\def\QQ{{\NZQ Q}}
\def\ZZ{{\NZQ Z}}

\def\frk{\mathfrak}               

\def\Phi{{\frk N}}

\def\opn#1#2{\def#1{\operatorname{#2}}} 
\opn\chara{char} \opn\length{\ell} \opn\pd{pd} \opn\rk{rk}
\opn\projdim{proj\,dim} \opn\injdim{inj\,dim} \opn\rank{rank}
\opn\depth{depth} \opn\grade{grade} \opn\height{height}
\opn\embdim{emb\,dim} \opn\codim{codim}

\opn\Tr{Tr} \opn\bigrank{big\,rank}
\opn\superheight{superheight}\opn\lcm{lcm}
\opn\trdeg{tr\,deg}
\opn\reg{reg} \opn\lreg{lreg} \opn\ini{in} \opn\lpd{lpd}
\opn\size{size}\opn{\mult}{mult}
\opn\div{div} \opn\Div{Div} \opn\cl{cl} \opn\Cl{Cl}
\opn\Spec{Spec} \opn\Supp{Supp} \opn\supp{supp} \opn\Sing{Sing}
\opn\Ass{Ass} \opn\Min{Min}
\opn\Ann{Ann} \opn\Rad{Rad} \opn\Soc{Soc}
\opn\Syz{Syz} \opn\Im{Im} \opn\Ker{Ker} \opn\Coker{Coker}
\opn\Am{Am} \opn\Hom{Hom} \opn\Tor{Tor} \opn\Ext{Ext}
\opn\End{End} \opn\Aut{Aut} \opn\id{id}

\opn\nat{nat}
\opn\pff{pf}
\opn\Pf{Pf} \opn\GL{GL} \opn\SL{SL} \opn\mod{mod} \opn\ord{ord}
\opn\Gin{Gin}
\opn\Hilb{Hilb}\opn\adeg{adeg}\opn\std{std}\opn\ip{infpt}
\opn\Pol{Pol}
\opn\sat{sat}
\opn\Var{Var}
\opn \ann{ann}
\opn\aff{aff} \opn\con{conv} \opn\relint{relint} \opn\st{st}
\opn\lk{lk} \opn\cn{cn} \opn\core{core} \opn\vol{vol}
\opn\link{link} \opn\star{star}
\opn\gr{gr}

\def\pot#1#2{#1[\kern-0.28ex[#2]\kern-0.28ex]}

\opn\dirlim{\underrightarrow{\lim}}
\opn\inivlim{\underleftarrow{\lim}}
\let\To=\longrightarrow
\def\Implies{\ifmmode\Longrightarrow \else
        \unskip${}\Longrightarrow{}$\ignorespaces\fi}
\def\implies{\ifmmode\Rightarrow \else
        \unskip${}\Rightarrow{}$\ignorespaces\fi}
\def\iff{\ifmmode\Longleftrightarrow \else
        \unskip${}\Longleftrightarrow{}$\ignorespaces\fi}

\let\:=\colon
\theoremstyle{plain}
\newtheorem{Theorem}{Theorem}[section]
\newtheorem{Lemma}[Theorem]{Lemma}
\newtheorem{Corollary}[Theorem]{Corollary}
\newtheorem{Proposition}[Theorem]{Proposition}
\newtheorem{Remark}[Theorem]{Remark}

\newtheorem{Conjecture}[Theorem]{Conjecture}

\let\epsilon\varepsilon
\let\phi=\varphi
\let\kappa=\varkappa

\def\qed{\ifhmode\textqed\fi
      \ifmmode\ifinner\quad\qedsymbol\else\dispqed\fi\fi}
\def\textqed{\unskip\nobreak\penalty50
       \hskip2em\hbox{}\nobreak\hfil\qedsymbol
       \parfillskip=0pt \finalhyphendemerits=0}
\def\dispqed{\rlap{\qquad\qedsymbol}}

\opn\dis{dis}
\def\pnt{{\raise0.5mm\hbox{\large\bf.}}}

\opn\Lex{Lex}

\begin{document}

\title{Bounds for Hilbert coefficients}

\author{J\"urgen Herzog and Xinxian Zheng}

\address{J\"urgen Herzog, Fachbereich Mathematik und
Informatik, Universit\"at Duisburg-Essen, Campus Essen, 45117
Essen, Germany} \email{juergen.herzog@uni-essen.de}

\address{Xinxian Zheng, Fachbereich Mathematik und
Informatik, Universit\"at Duisburg-Essen, Campus Essen, 45117
Essen, Germany} \email{xinxian.zheng@uni-essen.de}

\thanks{The second author is grateful for the financial support by DFG (Deutsche Forschungsgemeinschaft) during the preparation of this work}

\begin{abstract}
We compute the Hilbert coefficients of a graded module with pure resolution and discuss lower  and upper  bounds for these coefficients for arbitrary graded modules.
\end{abstract}
\maketitle
\section*{Introduction}
Let $K$ be a field, $S=K[x_1,\ldots, x_n]$ the polynomial ring in
$n$ variables,  and let $N$ be any graded $S$-module of dimension
$d$. Then for $i\gg 0$, the numerical function $H(N,i)=\sum_{j\leq
i}\dim_KN_j$ is a polynomial function of degree $d$, see \cite[4.1.6]{BH}.
In other words, there exists a polynomial $P_N(x)\in \ZZ[x]$ such
that
\[
H(N,i)=P_N(i) \quad \text{for all}\quad i\gg 0.
\]
The polynomial $P_N(x)$ is called the {\em Hilbert polynomial} of $N$.
It can be
written in the form

\begin{eqnarray*}
\label{formula}
P_N(x)=\sum_{i=0}^{d}(-1)^ie_{i}(N)\binom{x+d-i}{d-i}
\end{eqnarray*}
with integer coefficients $e_i(N)$, called the {\em Hilbert
coefficients} of $N$.

In the first section we will give explicit formulas for the $e_i(N)$
in case $N$ has a pure resolution. In the second section we use this
result and a conjecture of Boij and S\"oderberg \cite{BS} to get
conjectural lower and upper bounds for the Hilbert coefficients.  We
also discuss a few cases for which these  bounds hold. These bounds
generalize the conjectured bounds for the multiplicity, due to
Huneke, Srinivasan and the first author of this paper, see
\cite{HS}. A rather complete  survey of the multiplicity conjecture
can be found in \cite{HZ}. For more recent results we refer to
\cite{tim}, \cite{welker}, \cite{novik} and \cite{HiSi}.

\section{The Hilbert coefficients of a module with pure resolution}

Let $K$ be a field and $S=K[x_1,\ldots, x_n]$ the polynomial ring in $n$ variables, and let $N$ be a finitely generated graded $S$-module. We say $N$ has a pure resolution of type $(d_0,d_1,\ldots, d_s)$ if its minimal graded free $S$-resolution is of the form
\[
0\To S^{ \beta_s}(-d_{s})\To\cdots\To S^{\beta_1}(-d_{1})\To S^{\beta_0}(-d_0) \To 0.
\]
The main result of this section is

\begin{Theorem}
\label{main}
 Let $N$ be  a finitely generated graded Cohen-Macaulay $S$-module of codimension $s$ with pure resolution of type $(d_0,d_1,\ldots, d_s)$ with $d_0=0$. Then the Hilbert coefficients of $N$ are
\[
e_i(N)=\beta_0\frac{\prod_{j=1}^sd_j}{(s+i)!}\sum_{1\leq j_1\leq j_2\cdots \leq j_i\leq s}\; \prod_{k=1}^i(d_{j_k}-(j_k+k-1)), \quad i=0,\ldots, n-s.
\]
\end{Theorem}

\begin{proof}
We first recall a few facts about Hilbert series and multiplicities as described in \cite{BH}. The Hilbert series $H_N(t)=\sum_iH(N,i)t^i$ is a rational function of the form
\[
H_N(t)=\frac{Q_N(t)}{(1-t)^{d+1}},
\]
where $d=n-s$ is the dimension of $N$. The Hilbert coefficients $e_i=e_i(N)$ of $N$ can be computed according to the formula
\[
e_i=\frac{Q_N^{(i)}(1)}{i!}, \quad i=0,\ldots,d.
\]
On the other hand by using the additivity of Hilbert functions, the free resolution of $N$ yields the presentation
\[
H_N(t)=\frac{P_N(t)}{(1-t)^{n+1}} \quad \text{with}  \quad P_N(t) = \sum_{j=0}^s(-1)^j\beta_jt^{d_j}.
\]
Thus we see that $P_N(t)=Q_N(t)(1-t)^s$. This yields
\begin{eqnarray}
\label{ei}
e_i=(-1)^s\frac{P_N^{(s+i)}(1)}{(s+i)!},  \quad i=0,\ldots, d.
\end{eqnarray}
For any two integers $0\leq a\leq b$ we set
\[
g_a(b)=\sum_{1\leq i_1<i_2<\cdots <i_a\leq b}i_1i_2\cdots i_a.
\]
Then we have
\begin{eqnarray*}
P_N^{(s+i)}(1)&=&\sum_{j=0}^s(-1)^j\beta_j\prod_{k=0}^{s+i-1}(d_j-k)\\
&=& \sum_{j=0}^s(-1)^j\beta_j\sum_{k=0}^{s+i}(-1)^{s+i-k}g_{s+i-k}(s+i-1)d_j^k\\
&=& \sum_{k=0}^{s+i}(-1)^{s+i-k}g_{s+i-k}(s+i-1)\sum_{j=0}^s(-1)^j\beta_jd_j^k.
\end{eqnarray*}
Hence if we set $a_{k}=\sum_{j=0}^s(-1)^j\beta_jd_j^{k+s}$ for all $k\geq 0$ and observe that $\sum_{j=0}^s(-1)^j\beta_jd_j^{k}=0$ for all $k<s$ (see \cite{HS} where the proof of this fact is given in the cyclic case), we obtain together with (\ref{ei}) the following identities
\begin{eqnarray}
\label{identities}
(-1)^s(s+i)!e_i= \sum_{k=0}^i(-1)^{i-k}g_{i-k}(s+i-1)a_k, \quad i=0,\ldots, d.
\end{eqnarray}
\medskip
\noindent
In order to compute the $a_i$ we consider for each $i$  the following matrix
\[
B_i=
\begin{pmatrix}
\beta_1d_1 & \beta_2d_2 & \cdots & \beta_sd_s\\
\beta_1d_1^2 & \beta_2d_2^2 & \cdots & \beta_sd_s^2\\
\vdots & \vdots & & \vdots\\
\beta_1d_1^{s-1} & \beta_2d_2^{s-1} & \cdots & \beta_sd_s^{s-1}\\
\beta_1d_1^{s+i} & \beta_2d_2^{s+i} & \cdots & \beta_sd_s^{s+i}
\end{pmatrix}.
\]
Replacing the last column of $B_i$ by the alternating sum of its columns we obtain the matrix $B_i'$ for which $\det B_i'=(-1)^s\det B_i$ and whose last columns is the transpose of $(0,0,\ldots, a_i)$.
It follows that
\begin{eqnarray}
\label{ai}
a_i=(-1)^s\det B_i/\det C,
\end{eqnarray} where
\[
C=
\begin{pmatrix}
\beta_1d_1 & \beta_2d_2 & \cdots & \beta_{s-1}d_{s-1}\\
\beta_1d_1^2 & \beta_2d_2^2 & \cdots & \beta_{s-1}d_{s-1}^2\\
\vdots & \vdots & & \vdots\\
\beta_1d_1^{s-1} & \beta_2d_2^{s-1} & \cdots & \beta_{s-1}d_{s-1}^{s-1}
\end{pmatrix}.
\]
Note that $\det N=\beta_1\cdots \beta_{s-1}d_1\cdots d_{s-1}\det V(d_1,\cdots, d_{s-1})$, where $V(d_1,\cdots, d_{s-1})$ is the Vandermonde matrix for the sequence $d_1,d_2,\ldots, d_{s-1}$. Hence we obtain
\[
\det C= \beta_1\cdots \beta_{s-1}d_1\cdots d_{s-1}\prod_{1\leq i<j\leq s-1}(d_j-d_i).
\]
On the other hand we have
\[
\det B_i= \beta_1\cdots \beta_sd_1\cdots d_s \det
\begin{pmatrix} 1 & 1& \cdots & 1\\
d_1 & d_2 & \cdots  & d_s\\
\vdots & \vdots && \vdots \\
d_1^{s-2} & d_2^{s-2} & \cdots & d_s^{s-2}\\
d_1^{s+i-1} & d_2^{s+i-1} & \cdots & d_s^{s+i-1}
\end{pmatrix}.
\]
According to the subsequent Lemma \ref{expansion} we have
\[
\det \begin{pmatrix} 1 & 1& \cdots & 1\\
d_1 & d_2 & \cdots  & d_s\\
\vdots & \vdots && \vdots \\
d_1^{s-2} & d_2^{s-2} & \cdots & d_s^{s-2}\\
d_1^{s+i-1} & d_2^{s+i-1} & \cdots & d_s^{s+i-1}
\end{pmatrix}
=f_{i}(d_1,\ldots, d_s)\cdot\prod_{1\leq j<k\leq s}(d_k-d_j) ,
\]
where for each integer $k\geq 0$ we set
\[
f_k(g_1,\ldots,g_s)=\sum g_1^{c_1}\cdots g_s^{c_s}.
\]
Here the sum is taken over all integer vectors $c=(c_1,\ldots,c_s)$ with $c_i\geq 0$ for all $i$ and $|c|=\sum_{i=1}^sc_i=k$.

Thus by (\ref{ai}) we have
\[
a_i= (-1)^s\beta_sd_sf_i(d_1,\ldots,d_s)\prod_{j=1}^{s-1}(d_s-d_j).
\]
Now we use that fact that $\beta_s=\beta_0\prod_{j=1}^{s-1}d_j/\prod_{j=1}^{s-1}(d_s-d_j)$ (see \cite{HK} or \cite{BS}) and we obtain
\[
a_i=(-1)^s\beta_0d_1\cdots d_s f_i(d_1,\ldots,d_s).
\]
This result together with (\ref{identities}) yields the  formulas
\begin{eqnarray}
\label{nice}
e_i=\beta_0\frac{d_1\cdots d_s}{(s+i)!}\sum_{j=0}^i(-1)^{i-j}g_{i-j}(s+i-1)f_j(d_1,\ldots,d_s).
\end{eqnarray}
Expanding the products in the following sum
\[
\sum_{1\leq j_1\leq j_2\cdots \leq j_i\leq s}\; \prod_{k=1}^i(d_{j_k}-(j_k+k-1))
\]
yields
\[
\sum_{1\leq j_1\leq j_2\cdots \leq j_i\leq s}\; \prod_{k=1}^i(d_{j_k}-(j_k+k-1))=\sum_{j=0}^i(-1)^{i-j}g_{i-j}(s+i-1)f_j(d_1,\ldots,d_s).
\]
Hence the desired formulas for the $e_i$ follow from (\ref{nice}).
\end{proof}

It remains to prove

\begin{Lemma}
\label{expansion}
For all $k\geq s-1\geq 0$ one has
\[
\det
\begin{pmatrix}1& \ldots & 1\\
d_1 &\ldots & d_s\\
\vdots && \vdots \\
d_1^{s-2} &\ldots & d_s^{s-2}\\
d_1^k &\ldots & d_s^k\\
\end{pmatrix} =f_{k-s+1}(d_1,\ldots, d_s)\cdot \prod_{1\leq i<j\leq s}(d_j-d_i).
\]
\end{Lemma}

\begin{proof}
Given integers  $1\leq r\leq s$ and $k\geq s$ we define the matrix
\[
A_r^{(k)}=(a_{ij}^{(k)})_{i=1,\ldots, s-r+1\atop j=1,\ldots, s-r+1}
\]
with
\[
a_{ij}^{(k)}= \left\{\begin{array}{ll}
f_{i-1}(d_1,\ldots, d_{r-1}, d_{r-1+j}),& \text{for $i\leq s-r$, $j=1,\ldots, s-r+1$}\\
f_{k-r+1}(d_1,\ldots, d_{r-1}, d_{r-1+j}),& \text{for $i= s-r+1$, $j=1,\ldots, s-r+1$}.
\end{array} \right.
\]
Notice that $A_1^{(k)}$ is the matrix whose determinant we want to compute, while $A^{(k)}_s$
is the $1\times 1$-matrix with entry $f_{k-s+1}(d_1,\ldots,d_{s-1},d_s)$.

Next observe that for each integer $\ell>0$ and all $j>1$ one has
\[
f_\ell(d_1,\ldots, d_{r-1},d_{r-1+j})-f_{\ell}(d_1,\cdots, d_{r-1},d_r)=(d_{r-1+j}-d_r)\cdot f_{\ell-1}(d_1,\ldots, d_{r},d_{r-1+j}).
\]
Hence if subtract the first column from the other columns of $A_r^{(k)}$ and expand then this new matrix with respect to the first row (which is $(1,0,\cdots,0)$) we see that
\[
\det A_r^{(k)}=(d_{r+1}-d_r)(d_{r+2}-d_r)\cdots (d_s-d_r)\det A^{(k)}_{r+1}.
\]
Form this we obtain that
\[
\det A_1^{(k)}=\det A_s^{(k)}\cdot \prod_{1\leq i<j\leq s}(d_j-d_i)=f_{k-s+1}(d_1,\ldots,d_{s-1},d_s)\cdot \prod_{1\leq i<j\leq s}(d_j-d_i),
\]
as desired.
\end{proof}

For $i=0,1,2$ the formulas for the Hilbert coefficients read as follows:

\medskip
$e_0(N)=\beta_0\frac{\prod_{i=1}^sd_i}{s!}$;

\medskip
$e_1(N)=\beta_0\frac{\prod_{i=1}^sd_i}{(s+1)!}\sum_{i=1}^s(d_i-i)$;

\medskip
$e_2(N)=\beta_0\frac{\prod_{i=1}^sd_i}{(s+2)!}\sum_{1\leq i \leq j \leq s}(d_i-i)(d_j-j-1)$.

\medskip
In the special case that $N$ has a $d$-linear resolution, our formulas yield
\[
e_i(N)=\beta_0 {d+s-1 \choose s+i}{s+i-1\choose i}.
\]

\medskip

\begin{Remark}
\label{remark}
{\em The assumption made in Theorem \ref{main} that $d_0$ should be zero,  is not essential. It is only made to simplify the formulas for the Hilbert coefficients. While for the multiplicity we have $e_0(N)=e_0(N(a))$ for any shift $a$, the other Hilbert coefficients transform as follows:  if $N$ has a pure resolution of type $(d_0,d_1,\ldots, d_s)$, then $N(d_0)$ has pure resolution of type $(0, d_1-d_0,\ldots, d_s-d_0)$ whose Hilbert coefficient we know by Theorem~\ref{main}.

On the other hand we have $P_{N}(x)=P_{N(d_0)}(x+d_0)$,  from which one deduces that
\[
e_i(N)=\sum_{j=0}^i{d_0-1+i-j\choose d_0-1}e_{i}(N(d_0)) \quad \text{for} \quad i=0,\ldots, n-s.
\]
}
\end{Remark}

\section{Upper and lower bounds}

Given a sequence $d_1,d_2, \ldots,d_s$ of integers. We set
\[
h_i(d_1,\ldots, d_s)= \sum_{1\leq j_1\leq j_2\cdots \leq j_i\leq s}\; \prod_{k=1}^i(d_{j_k}-(j_k+k-1))
\]
for $i=0,\ldots, n-s$. This definition will  simplify notation in the following discussions.

Let $N$ be any finitely generated graded Cohen-Macaulay  $S$-module of projective dimension $s$ and graded Betti numbers $\beta_{ij}$. For each $i=1,\ldots s$, the minimal and maximal shifts of $N$ in homological degree $i$ are defined by  $m_i=\min\{j\: \; \beta_{ij}\neq 0\}$ and $M_i=\max\{j\: \; \beta_{ij}\neq 0\}$.

In case $N$ is generated in degree $0$ and  has a pure resolution of type $(d_1,\ldots,d_s)$, we have $m_i=M_i=d_i$ for all $i$, and Theorem \ref{main} tells us that
\[
e_i(N)=\beta_0\frac{d_1d_2\cdots d_s}{(s+i)!}h_i(d_1,\ldots,d_s) \quad \text{for $i=0,1,\ldots, n-s$}.
\]
In analogy to the so-called multiplicity conjecture we now state
\begin{Conjecture}
\label{1}
Let $N$ be a finitely generated graded Cohen-Macaulay $S$-module of codimension $s$ generated in degree $0$. Then
\[
\beta_0\frac{m_1m_2\cdots m_s}{(s+i)!}h_i(m_1,\ldots,m_s)\leq e_i(N)\leq \beta_0\frac{M_1M_2\cdots M_s}{(s+i)!}h_i(M_1,\ldots,M_s)
\]
for $i=0,1,\ldots, n-s$.
\end{Conjecture}

Next we recall a conjecture of Boij and S\"oderberg \cite{BS}: for any strictly increasing sequences of integers
  $m=(m_0,\dots,m_s)$ and $M =
  (M_0,\dots,M_s)$, let $V_{m,M}$ be the vector space over the
  rational numbers of all matrices $(\beta_{i,j})$ such that:
  \begin{enumerate}
  \item[(a)] $\beta_{i,j}$ is a solution to the system of linear
    equations
    \[ \label{vSpaceEq} \left\{
    \begin{array}{lcc}
      \sum_{i,j} (-1)^i \beta_{i,j} &=& 0 \\
      \sum_{i,j} (-1)^i j\beta_{i,j} &=& 0 \\
      & \vdots & \\
      \sum_{i,j} (-1)^i j^{s-1}\beta_{i,j} &=& 0
    \end{array}\right.
\]

  \item[(b)] $\beta_{i,j}=0$ whenever $j<m_i$ or $j>M_i$ (or $i <
    0$ or $i>s$).
  \end{enumerate}

Note that the graded Betti numbers of any graded Cohen-Macaulay module $N$ of codimension $s$ satisfies condition (a). Moreover if the maximal and minimal shifts of $N$ are the numbers $m_i$ and $M_i$, then this Betti diagram belongs to $V_{m,M}$. The set of Betti diagrams in $V_{m,M}$ is denoted by $B_{m,M}$. It is an additively closed subset of $V_{m,M}$.

To each $(\beta_{ij})\in B_{m,M}$ we assign the normalized Betti diagram $(\bar{\beta}_{ij})=(\beta_{ij}/\beta_0)$ and define the subset $\bar{B}_{m,M}=\{(\bar{\beta}_{ij})\:\; (\beta_{ij})\in B_{m,M}\}$ of $V_{m,M}$. The set $\bar{B}_{m,M}$ is closed under convex combinations with rational coefficients.

For any strictly increasing sequence of integers $d=(d_0,d_1,\dots,
  d_s)$, the matrix $\pi(d)$ defined by
$$
\pi(d)_{i,j} =
\begin{cases}(-1)^{i+1} \prod_{\substack{k \neq i \\k \neq 0}}
  \frac{d_k-d_0}{d_k-d_i}
  \text{ if $j = d_i$,} \\
  0 \text{ if $j \neq d_i$},
\end{cases}
$$
is called a pure diagram.

For any two strictly increasing sequences  of integers $m=(m_0,\ldots, m_s)$ and $M=(M_0,\ldots, M_s)$ we denote by $\Pi_{m,M}$ the set of all pure diagrams in $V_{m,M}$. Note that $\Pi_{m,M}$ is just the set of pure diagrams $\pi(d)$ with $m_i\leq d_i\leq M_i$ for all $i$.

Now we have the following

\begin{Conjecture}[Boij, S\"oderberg]
\label{BS} $\bar{B}_{m,M}$ is the convex hull of $\Pi_{m,M}$.
\end{Conjecture}

If it happens that $N$ is a graded Cohen-Macaulay module generated in degree $0$ with pure resolution of type $d=(d_1,\ldots,d_s)$, then the normalized Betti diagram of $N$ is just $\pi(d)$ (with $d_0=0$),  as follows from \cite{HK} (see also \cite{BS}). Hence we define for $i=0,\ldots,n-s$,  the Hilbert coefficients of a pure diagram $\pi(d)$ for which  $d_0=0$  as
\[
e_i(\pi(d))= \frac{d_1d_2\cdots d_s}{(s+i)!}h_i(d_1,\ldots,d_s),
\]
no matter whether or not $d$ is the type of a Cohen-Macaulay module with pure resolution.

The following observation justifies our conjecture.

\begin{Proposition}
\label{justify} Conjecture \ref{BS} implies Conjecture \ref{1}.
\end{Proposition}

\begin{proof} Let $N$ is a graded Cohen-Macaulay module of codimension $s$  generated in degree $0$, and let $D$ be the normalized Betti diagram of $N$. Let $m=(m_1,\ldots,m_s)$ and $M=(M_1,\ldots,M_s)$ be the sequences of minimal and maximal shifts of $D$. Assuming Conjecture~\ref{BS} we have
\[
D=\sum_{\pi(d)\in \Pi_{m,M}}c_{\pi(d)}\pi(d) \quad \text{with}\quad c_{\pi(d)}\in\QQ \quad \text{and} \sum_{\pi(d)\in \Pi_{m,M}}c_{\pi(d)}=1.
\]
It follows that
\begin{eqnarray}
\label{ein}
e_i(N)=\beta_0\cdot\sum_{\pi(d)\in \Pi_{m,M}}c_{\pi(d)}e_i(\pi(d))
\end{eqnarray}
Let  $\prod_{k=1}^i(d_{j_k}-(j_k+k-1))$ be one of the summands in  $h_i(d)$. We claim that either $\prod_{k=1}^i(d_{j_k}-(j_k+k-1))=0$, or else $d_{j_k}-(j_k+k-1)>0$ for $k=1,\ldots,i$. The claim will then imply that \begin{eqnarray}
\label{eipi}
e_i(\pi(d))\leq e_i(\pi(d'))
\end{eqnarray}
whenever we have $d_i\leq d_i'$  for  $i=1,\ldots,s.$

In order to prove the claim  suppose that $\prod_{k=1}^i(d_{j_k}-(j_k+k-1))\neq 0$.  Since $d_i\geq i$ for all $i$, we must then have that $d_{j_1}-j_1> 0$.  Assume that not all factors $d_{j_k}-(j_k+k-1)$ are positive and let
$\ell$ be the smallest integer with $d_{j_\ell}-(j_\ell+\ell-1)<0$. Then $\ell>1$ and  $d_{j_{\ell-1}}-(j_{\ell-1}+\ell-2)>0$. It follows that
\[
d_{j_{\ell-1}}-(j_{\ell-1}+\ell-2)-(d_{j_\ell}-(j_\ell+\ell-1))\geq 2,
\]
equivalently
\[
j_\ell-j_{\ell-1}\geq d_{j_\ell}-d_{j_{\ell-1}}+1.
\]
This is a contradiction, since $d_1<d_2<\cdots < d_s$.

Now (\ref{ein}) and (\ref{eipi}) imply that
\begin{eqnarray*}
e_i(\pi(m))&\leq &\min\{e_i(\pi(d))\:\; \pi(d)\in \Pi_{m,M}\}\\
&\leq &\frac{e_i(N)}{\beta_0}\leq \max\{e_i(\pi(d))\:\; \pi(d)\in \Pi_{m,M}\}=e_i(\pi(M)),
\end{eqnarray*}
as desired.
\end{proof}

Conjecture \ref{BS} is proved in several cases by Boij and S\"oderberg, and hence also proves our conjecture in these cases. We single out two such cases.

\begin{Corollary} Let  $N$ be  a Cohen-Macaulay $S$-module of codimension two, generated
  in degree $0$, or let  $N=S/I$ where $I$ is a Gorenstein ideal of codimension 3.  Then the bounds for the Hilbert coefficients given in Conjecture \ref{1} hold.
\end{Corollary}

\end{document}